\documentclass[12pt]{amsart}

\usepackage{amssymb,mathtools}
\usepackage{amsrefs}

\makeatletter
\@namedef{subjclassname@2010}{%
  \textup{2010} Mathematics Subject Classification}
\makeatother

\newtheorem{theorem}{Theorem}

\newtheorem*{unnumberedtheorem}{Theorem}

\newcommand{\bt}{\begin{theorem}}
\newcommand{\et}{\end{theorem}}
\newtheorem{lemma}[theorem]{Lemma}
\newcommand{\bl}{\begin{lemma}}
\newcommand{\el}{\end{lemma}}
\newtheorem{corollary}[theorem]{Corollary}
\newcommand{\bc}{\begin{corollary}}
\newcommand{\ec}{\end{corollary}}
\newcommand{\beq}{\begin{equation}}
\newcommand{\eeq}{\end{equation}}
\newcommand{\benum}{\begin{enumerate}}
\newcommand{\eenum}{\end{enumerate}}
\newcommand{\N}{\ensuremath{ \mathbb N }}
\newcommand{\Z}{\ensuremath{\mathbb Z}}

\newcommand{\R}{\ensuremath{\mathbb R}}
\newcommand{\mcA}{\ensuremath{ \mathcal A}}

\newcommand{\mcQ}{\ensuremath{ \mathcal Q}}
\newcommand{\mcX}{\ensuremath{ \mathcal X}}

\DeclareMathOperator{\card}{\text{card}}

\newcommand{\floor}[1]{\left\lfloor #1 \right\rfloor}

\newcommand{\fp}[1]{\left\{ #1 \right\}}
\newcommand{\tf}[1]{\left[\!\left[ #1 \right]\!\right]}
\newcommand{\dtf}[1]{\left[\!\!\left[ #1 \right]\!\!\right]}

\title[Frac. parts of roots]{The sequence of fractional parts of roots}
\author[K. O'Bryant]{Kevin O'Bryant}
\address{Department of Mathematics\\
College of Staten Island (CUNY), Staten Island, NY 10314\\
and CUNY Graduate Center, New York, NY 10016}
\email{kevin@member.ams.org}
\thanks{This work grew out of discussions with Melvyn B. Nathanson, to whom I
am profoundly grateful for his problems, advice, and ideas.}
\thanks{Support for this project was provided by a PSC-CUNY Award, jointly funded by The Professional Staff Congress and The City University of New York.}
\thanks{This work will appear in \emph{Acta Arithmetica}, to whose
  referee and editor the author is grateful.}

\date{\today}

\begin{document}

\begin{abstract}
We study the function $M_\theta(n)=\floor{1/\fp{\theta^{1/n}}}$, where $\theta$ is a positive real number, $\floor{\cdot}$ and $\fp{\cdot}$ are the floor and fractional part functions, respectively. Nathanson 
proved, among other properties of $M_\theta$, that if $\log\theta$ is rational, then for all but finitely many positive integers $n$, $M_\theta(n)=\floor{n/\log\theta-1/2}$. We extend this by showing that, without condition on $\theta$, all but a zero-density set of integers $n$ satisfy  $M_\theta(n)=\floor{n/\log\theta-1/2}$. Using a metric result of Schmidt, 
 we show that almost all $\theta$ have asymptotically $(\log\theta \log x)/12$ exceptional $n \leq x$. 
 Using continued fractions, we produce uncountably many $\theta$ that have only finitely many exceptional $n$, and also give uncountably many explicit $\theta$ that have infinitely many exceptional $n$.
\end{abstract}


\maketitle

\section{Introduction}
The author finds the identity (valid for any nonzero integer $n$)
  \begin{equation}\label{equ:root2}
   \floor{\frac{1}{e^{\sqrt{2}/n} -1 }} = \floor{\frac{n}{\sqrt{2}}-\frac12}
  \end{equation}
breathtaking. Even more perplexing is that the similar expression (see~\cite{OEIS})
  \begin{equation}\label{equ:log2}
   \floor{\frac{1}{2^{1/n} -1 }} = \floor{\frac{n}{\log(2)}-\frac12}
  \end{equation}
holds for integers $1<n<777\,451\,915\,729\,368$, but fails at both of the given endpoints.

This identity and near-identity arise in our study of the sequence of fractional parts of roots, following Nathanson~\cite{Nathanson.Mtheta}. The distribution of $(\fp{\theta^n})_{n\geq1}$, where $\theta>1$, has been the object of much study~\cite{Drmota.Tichy} but remains enigmatic except for a few peculiar $\theta$. The sequence $(\fp{\theta^{1/n}})_{n \geq 1}$ has been thought too simple to warrant study: trivially, for $\theta>1$ one has $\theta^{1/n}>1$ and $\theta^{1/n}\to 1$, and so $\fp{\theta^{1/n}}\to 0$. Nevertheless, Nathanson found interesting phenomena in the regularity with which this convergence takes place. He introduced and derived the basic properties of
  \[
   M_\theta(n) \coloneqq \floor{\frac{1}{\fp{\theta^{1/n}}}},
  \]
and identified symmetries that allow one to assume without loss of generality that $\theta>1$ and that the integer $n$ is positive. Surprisingly, he proved that for any real $\theta>1$ and integer $n> \log_2\theta$, either $M_\theta(n)=\floor{n/\log\theta-1/2}$ or $M_\theta(n)=\floor{n/\log\theta+1/2}$; moreover, if $\log\theta$ is rational, then $M_\theta(n)=\floor{n/\log\theta-1/2}$ for all sufficiently large $n$.

We will show that the set
  \begin{equation}\label{equ:atypicalintro}
   \left\{ n \in \N :  M_\theta(n) \neq \floor{\frac{n}{\log\theta}-\frac12} \right\}
  \end{equation}
has density 0 for all $\theta>1$, and for almost all $\theta>1$ has counting function asymptotic to $\frac{\log\theta}{12}\,\log n$. For $\theta<e^6\approx 400$, we give criteria for \eqref{equ:atypicalintro} to be finite or infinite in terms of the continued fraction expansions of $1/\log\theta$ and $2/\log\theta$. As a consequence, we are able to give explicit $\theta$ for which \eqref{equ:atypicalintro} is empty and is infinite. As mentioned above, Nathanson proved that for $\theta=e^{p/q}$, \eqref{equ:atypicalintro} is finite; we give another proof of this below that gives an explicit bound on the size in terms of $p$ and $q$.

In the final sections of this article, we discuss the two displayed equations at the beginning of this introduction, and update Nathanson's list of open problems for $M_\theta(n)$.

\section{Conventions, Results, Strategy}
The set of positive integers is denoted \N. Throughout, we assume that $\theta>1$ and that $n$ is a positive integer. 
If $n>\log_2\theta$, then $1<\theta^{1/n} < 2$, and so $\fp{\theta^{1/n}}= \theta^{1/n}-1$. Set
  \[
   M'_\theta(n) \coloneqq \floor{\frac{1}{\theta^{1/n}-1}},
  \]
so that $M_\theta(n)=M'_\theta(n)$ if $n>\log_2\theta$. Although we don't use it here, this sort of expression arises \cites{Golomb.Hales,Nathanson.Shat} 
in the following way.
$M'_\theta(n)$ is the largest integer $N$ such that $\theta N^n \leq (N+1)^n$; and $M'_\theta(n)$ is the largest integer $N$ such that $(1+\frac 1N)^n \geq \theta$. 
We call the elements of
  \[
   \mcA_\theta \coloneqq \left\{ n \in \N : M'_\theta(n) \neq \floor{\frac{n}{\log\theta}-\frac12}\right\}
  \]
the {\em atypical} numbers, terminology which we will justify later. Nathanson proved the following result, albeit in different notation. 
\begin{theorem}\label{FPCF:theorem:Plus or Minus}
   If $n>\log_2\theta$, then either
   \[
      M_\theta(n) = \floor{\frac{n}{\log\theta}-\frac12} \qquad\text{ and }\qquad n\not\in \mcA_\theta
   \]
   or
   \[
      M_\theta(n) = \floor{\frac{n}{\log\theta}+\frac12} \qquad\text{ and }\qquad n\in \mcA_\theta.
   \]
\end{theorem}

This shows that understanding $M_\theta(n)$ for nonsmall $n$ is equivalent to understanding $\mcA_\theta$. Our results are presented as properties of $\mcA_\theta$. We state the theorems here, although we define some of the terminology, such as that relating to density and continued fractions, in the proofs.

\begin{theorem}[Nathanson]\label{FPCF:theorem:rational}
 If $\log\theta=p/q>1$ is rational, then
         \[\mcA_\theta \subseteq [1,\frac{p^2}{6q}).\]
\end{theorem}

\begin{theorem}\label{FPCF:theorem:zero density}
  For all $\theta>1$, $\mcA_\theta$ has density 0.
\end{theorem}

\begin{theorem}\label{FPCF:theorem:almost all}
  For almost all $\theta>1$,
         \[ \big| \mcA_\theta \cap [1,n] \big| \sim \frac{\log\theta}{12} \log n.\]
\end{theorem}

\begin{theorem}\label{FPCF:theorem:explicit theta with no atypical numbers}
   Let $a_i$ be positive integers with $a_{2k}=1$ for $k\geq0$. Set $\ell$ to be the irrational number with simple continued fraction $[a_0;a_1,a_2,\dots]$, and set $\theta=e^{2/\ell}$. Then $\mcA_\theta=\emptyset$. In particular, if $c\in\N$ and $\theta=e^{-c+\sqrt{c(c+4)}}$, then $\mcA_\theta$ is empty.
\end{theorem}

\begin{theorem}\label{FPCF:theorem:explicit theta with many atypical numbers}
   Let $a_i$ be positive integers with $a_0=0$, $a_1=2$, $a_{2k}=4$ for all $k\geq1$. Set $\ell$ to be the irrational with simple continued fraction $[a_0;a_1,a_2,\dots]$, and set $\theta=e^{2/\ell}$. Then $\mcA$ is infinite. In particular, if $c\in\N$ and $\theta=e^{4-c+\sqrt{c(c+1)}}$, then $\mcA_\theta$ is infinite.
\end{theorem}

The last two theorems give explicit uncountable families of $\theta$ with $\mcA_\theta$ empty and infinite, and also draw attention to even more explicit countable subfamilies. The simplest examples are that $\mcA_{e^{\sqrt5-1}}$ is empty and $\mcA_{e^{2\sqrt5}}$ is infinite. The actual results proved are inequalities on the partial quotients of the continued fraction, and the specific $a_i$ given in these theorems are not the only $a_i$ that satisfy the inequalities. Our countable families consist entirely of transcendental numbers; we do not know if there is an algebraic $\theta$ with $\mcA_\theta=\emptyset$, nor if there is an algebraic $\theta$ with $\mcA_\theta$ infinite.

We now outline our approach. We first obtain an asymptotic expansion
  \[
   \frac{1}{\theta^{1/n}-1} = \frac{n}{\log\theta}-\frac12 + f\left(\frac{\log\theta}{n}\right)
  \]
for a very small positive function $f$. The floor of the left hand side is $M'_\theta(n)$, and the floor of the right hand side is $\floor{n/\log\theta-1/2}$ unless $n/\log\theta-1/2$ is within $f(\log\theta/n)$ of an integer. We are thus led to a nonhomogeneous diophantine approximation problem that we can partially handle with continued fractions. In particular, we will need to know the simple continued fractions of both $1/\log\theta$ and $2/\log\theta$. By defining $\theta$ through the continued fraction of $2/\log\theta$ we are able to set, or at least control, the size of $\mcA_\theta$.

\section{Bernoulli numbers and $n$th roots}

We use the generating function for the sequence
$(B_k)_{k=0}^{\infty}$ of Bernoulli numbers to obtain an asymptotic expansion of $M'_\theta(n)$.
For $|t| < 2\pi$,
\beq   \label{FPCF:Bern}
\frac{t}{e^t-1} = \sum_{k=0}^{\infty} \frac{B_k}{k!} t^k
 = 1 -\frac{1}{2}t + \frac{1}{12}t^2 - \frac{1}{720}t^4
 + \sum_{k=3}^{\infty} \frac{B_{2k}}{(2k)!} t^{2k}.
\eeq
For $t > 0$, we define the function
\beq  \label{FPCF:f(t)}
f(t) = \frac{1}{e^t-1} -  \frac{1}{t} + \frac{1}{2}.
\eeq

\bl      \label{FPCF:lemma:f-estimates}
For $t > 0$, the function $f(t)$
is strictly increasing,
$\lim_{t\rightarrow 0^+} f(t) = 0$, 
and
$\lim_{t\rightarrow \infty} f(t) = 1/2.$
If $0 < t <     1$, then
\beq  \label{FPCF:f-ineq}
\frac{t}{12} - \frac{t^3}{720} < f(t) < \frac{t}{12}.
\eeq
\el

\begin{proof}
The function $f(t)$ is strictly increasing (because $f'(t) > 0$),
$\lim_{t\rightarrow 0^+} f(t) = 0$ (apply l'H\^{o}pital's rule twice),
and
$\lim_{t\rightarrow \infty} f(t) = 1/2.$

For $0 < t < 2\pi$, we have the power series
\beq  \label{FPCF:powerseries}
f(t) = \frac{1}{12}t - \frac{1}{720}t^3 + \sum_{k=3}^{\infty} \frac{B_{2k}}{(2k)!} t^{2k-1}.
\eeq
The Bernoulli numbers satisfy the classical identity (\cite{Euler}, \cite{rade73}*{formula~(9.1)})
\[
\frac{B_{2k}}{(2k)!} =  \frac{2(-1)^{k-1}}{(2\pi)^{2k}} \sum_{n=1}^{\infty} \frac{1}{n^{2k}}.
\]
It follows that for $0 < t < 1$ the sequence
\[
 \left( \frac{|B_{2k}|}{(2k)!} t^{2k-1} \right)_{k=1}^{\infty}
\]
is strictly decreasing and tends to 0,
hence~\eqref{FPCF:powerseries} is an alternating series
and~\eqref{FPCF:f-ineq} follows.
This completes the proof.
\end{proof}

\begin{lemma}     \label{FPCF:lemma:FirstEquiv}
Either
  \[
    M'_{\theta}(n) = \floor{ \frac{n}{\log\theta}  - \frac{1}{2} }
    \text{ or }
    M'_{\theta}(n) = \floor{ \frac{n}{\log\theta}  + \frac{1}{2} },
  \]
and $M'_{\theta}(n) = \floor{n/\log \theta + 1/2}$
if and only if
\beq     \label{FPCF:BadFracPart2}
\frac{1}{2} -  f\left( \frac{\log\theta}{n}  \right)
\leq \fp{ \frac{n}{\log\theta} } <  \frac{1}{2} .
\eeq
\end{lemma}

Note that Theorem~\ref{FPCF:theorem:Plus or Minus} is a direct consequence of Lemma~\ref{FPCF:lemma:FirstEquiv}.

\begin{proof}
By the definition of $f$, we have
\begin{align*}
 \frac{1}{\theta^{1/n}-1}
    &= \frac{n}{\log\theta} - \frac{1}{2} + f\left( \frac{\log\theta}{n}  \right) \\
    &= \floor{ \frac{n}{\log\theta} - \frac{1}{2} } + \fp{ \frac{n}{\log\theta} - \frac{1}{2} } + f\left( \frac{\log\theta}{n}  \right)
\end{align*}
and so
\[
M'_{\theta}(n) = \floor{ \frac{n}{\log\theta} - \frac{1}{2} }
+ \floor{ \fp{ \frac{n}{\log\theta} - \frac{1}{2} }
 + f\left( \frac{\log\theta}{n}  \right) }.
\]
As fractional parts are between 0 and 1 while $f$ is between $0$ and $1/2$, we can now see that either
$M'_\theta(n) = \floor{n/\log\theta-1/2}$ or $M'_\theta(n)=\floor{n/\log\theta-1/2}+1$, which is the first claim of this lemma.
Moreover, $M'_{\theta}(n) = \floor{n/\log \theta + 1/2}$
if and only if
\[
1 \leq \fp{ \frac{n}{\log\theta} - \frac{1}{2} }
 + f\left( \frac{\log\theta}{n}  \right)  < 2.
\]
By Lemma~\ref{FPCF:lemma:f-estimates}, for $t > 0$ we have  $0 < f\left(t \right) < 1/2$, and if $1/2<\fp{t-1/2}$ then $\fp{t-1/2}=\fp{t}+1/2$, and so
\[
\frac{1}{2} < 1 -  f\left( \frac{\log\theta}{n}  \right)
\leq \fp{ \frac{n}{\log\theta} - \frac{1}{2} }
= \fp{ \frac{n}{\log\theta} } + \frac{1}{2} < 1.
\]
This implies~\eqref{FPCF:BadFracPart2}.

Conversely, inequality~\eqref{FPCF:BadFracPart2}
implies that
\[
1 \leq \fp{ \frac{n}{\log\theta} } + \frac{1}{2}
+   f\left( \frac{\log\theta}{n}  \right)
=  \fp{ \frac{n}{\log\theta} - \frac{1}{2} }
+   f\left( \frac{\log\theta}{n} \right) < \frac{3}{2} < 2.
\]
This completes the proof.
\end{proof}

\bl     \label{FPCF:lemma:FracBound}
If $0 \leq a < b \leq 1$, then
\begin{multline*}
\left\{ t \in \R: \frac{a}{2} \leq \fp{t} < \frac{b}{2}   \right\}
  = \\
\{t \in \R:  a \leq \fp{2t} < b\}
 \setminus
 \left\{ t \in \R: \frac{a+1}{2} \leq \fp{t} < \frac{b+1}{2}   \right\}.
\end{multline*}
\el

\begin{proof}
If $ \fp{t} < \frac{1}{2}$, then $\fp{2t} = 2\fp{t}$ and so $a \leq \fp{2t} < b$ if and only if $a/2 \leq \fp{t} < b/2$.
If $ \fp{t} \geq \frac{1}{2}$, then $\fp{2t} = 2\fp{t}-1$ and so $a \leq \fp{2t} < b$ if and only if $(a+1)/2 \leq \fp{t} < (b+1)/2$.
Thus,
\begin{align*}
\{t \in \R: &  a \leq \fp{2t} < b \} \\
=  & \left\{t \in \R: \fp{t} < \frac{1}{2} \text{ and }
  \frac{a}{2}  \leq \fp{t} < \frac{b}{2}  \right\} \\
& \hspace{2cm}  \cup \left\{t \in \R: \fp{t} \geq \frac{1}{2}  \text{ and } \frac{a+1}{2}
  \leq \fp{t} < \frac{b+1}{2}   \right\} \\
= & \left\{t \in \R: \frac{a}{2}  \leq \fp{t} < \frac{b}{2}  \right\}
 \cup \left\{t \in \R: \frac{a+1}{2} \leq \fp{t} < \frac{b+1}{2}   \right\}.
\end{align*}
The Lemma follows from the observation that the two sets on the right side of this equation are disjoint.
\end{proof}

Combining Lemmas~\ref{FPCF:lemma:FirstEquiv}  and~\ref{FPCF:lemma:FracBound} proves the following result.

\begin{lemma}      \label{FPCF:lemma:SecondEquiv}
We have $n\in\mcA_\theta$ if and only if both
   \[
   \fp{ \frac{2n}{\log \theta} }
   \geq 1 - 2f\left(  \frac{\log \theta}{n} \right)
   \]
and
   \[
   \fp{ \frac{n}{\log \theta} }
   < 1 - f\left(  \frac{\log \theta}{n} \right).
   \]
\end{lemma}

\section{Proofs of Theorems~\ref{FPCF:theorem:rational},~\ref{FPCF:theorem:zero density}, and~\ref{FPCF:theorem:almost all}}

Nathanson proved that $\mcA_{e^{p/q}}$ is finite; while his proof is different from the one here, it could also be pushed to give this bound.
\begin{proof}[Proof of Theorem~\ref{FPCF:theorem:rational}.]
Assume that $n\geq p^2/(6q)$. Then by Lemma~\ref{FPCF:lemma:f-estimates}
  \[ f\left(\frac{\log\theta}{n}\right)<\frac{\log\theta}{12n} \leq \frac{1}{2p},\]
and so there is no rational strictly between $1/2-f(\log\theta/n)$ and $1/2$ with denominator $p$. But clearly $\fp{n/\log\theta}$ has denominator $p$, and so Lemma~\ref{FPCF:lemma:FirstEquiv} tells us that $n\not\in \mcA_\theta$; that is, $\mcA_\theta\subseteq[1,p^2/(6q))$.
\end{proof}

The set $\mcX$ of positive integers has \emph{density $\epsilon$} if
  \[
    \lim_{N\rightarrow \infty}  \frac{\card\left(\left\{ x\in \mcX : x \leq N \right\} \right)}{N}
  \]
exists and is equal to $\epsilon$. We use the following results concerning density. If a set $\mcX$ is a subset of a set with density $\epsilon$ for every small $\epsilon>0$, then $\mcX$ has density 0. Let $0\leq a<b<1$ and let $\alpha$ be any irrational; the set $\left\{ n \in \N : a \leq \fp{n\alpha} < b\right\}$ has density $b-a$.

\begin{proof}[Proof of Theorem~\ref{FPCF:theorem:zero density}.]
If $\log\theta$ is rational, then by Theorem~\ref{FPCF:theorem:rational}, we know that $\mcA_\theta$ is finite. As finite sets have density 0, this case is handled.

Now suppose that $\log\theta$ is irrational. Take small $\epsilon>0$, and take $n_0=\log\theta/(12\epsilon)$ so that for all $n>n_0$ we have $f(\log\theta/n) < \frac{\log\theta}{12 n} <\epsilon$. Lemma~\ref{FPCF:lemma:FirstEquiv} now implies that
  \[
   \mcA_\theta \subseteq \left\{n\in\N : \frac12-\epsilon \leq \fp{\frac{n}{\log\theta}} < \frac 12 \right\}.
  \]
Since $\log\theta$ is irrational, $\fp{n/\log\theta}$ is uniformly distributed, and so $\mcA_\theta$ is contained in a set with density $\epsilon$. As $\epsilon$ was arbitrary, it follows that $\mcA_\theta$ has density 0.
\end{proof}

The main tool for Theorem~\ref{FPCF:theorem:almost all} is a result of Schmidt~\cite{schm64}*{Theorem 1}, which we state a special case of below. Let $\tf{P}$ be 1 if $P$ is true, and 0 if $P$ is false, and let $|I|$ be the length of the interval $I$.
\begin{unnumberedtheorem}[Schmidt]
   Let $I_1\supseteq I_2\supseteq\dots$ be a nested sequence of subintervals of $[0,1)$, $\epsilon>0$, then for almost all $\alpha$
      \[
         \sum_{k=1}^n \tf{ \fp{k\alpha} \in I_k} = \sum_{k=1}^n |I_k| + O\left( (\sum_{k=1}^n |I_k|)^{1/2+\epsilon}\right).
      \]
\end{unnumberedtheorem}

\begin{proof}[Proof of Theorem~\ref{FPCF:theorem:almost all}.]
   We apply Schmidt's theorem with $\alpha$ replaced by $1/\log\theta$, take $x>1$ and intervals
      \[
         I_k = \left[ \frac12 - f\left(\frac{\log x}{k}\right), \frac12 \right),
      \]
   which are properly nested since $f(t)$ is increasing for $t>0$.
   Since $f(t)=t/12+O(t^3)$ (as $t\to0$), as $n\to\infty$ we have
      \[
         \sum_{k=1}^n |I_k| = \sum_{k=1}^n f\left(\frac{\log x}{k}\right)
            = \sum_{k=1}^n \left(\frac{\log x}{12k}+O(1/k^3)\right)
            = \frac{\log x}{12} \log n + O(1).
      \]

   By Lemma~\ref{FPCF:lemma:FirstEquiv}, for $\theta<x$,
      \begin{align*}
         \mcA_\theta&=\left\{ k : \fp{\frac{k}{\log\theta}} \in [\frac12-f\left(\frac{\log\theta}{k}\right),\frac12 ) \right\} \\
            &\subseteq \left\{ k : \fp{\frac{k}{\log\theta}} \in [\frac12-f\left(\frac{\log x}{k}\right),\frac12 ) \right\} \\
            &= \left\{ k : \fp{\frac{k}{\log\theta}} \in I_k \right\},
      \end{align*}
   and so 
      \[ \left| A_\theta \cap [1,n] \right| \leq \sum_{k=1}^n \dtf{\fp{\frac{k}{\log \theta}} \in I_k}.\]
   Similarly, for $\theta>x$
      \[ \left| A_\theta \cap [1,n] \right| \geq \sum_{k=1}^n \dtf{\fp{\frac{k}{\log \theta}} \in I_k}.\]
      
Set
      \[ g(\theta) \coloneqq \limsup_{n\to\infty} \frac{\big| A_\theta\cap[1,n] \big|}{\log n}, \]
which must be Lebesgue measurable since its definition makes no appeal to the axiom of choice. One may verify the measurability of $g$ more directly by observing that, for fixed $n$, the preimages of $\theta \mapsto A_\theta \cap [1,n]$ are unions of half-open intervals, and so each $\theta \mapsto  \frac{| A_\theta\cap[1,n] |}{\log n}$ is a simple measurable function, and so $g(\theta)$ is the $\limsup$ of a sequence of simple measurable functions, and so is itself measurable.

Schmidt's theorem implies: for all $x>1$, almost all $\theta<x$ satisfy
      \[
         g(\theta) \leq \frac{\log x}{12},
      \]
   and almost all $\theta>x$ satisfy
      \[
         g(\theta) \geq \frac{\log x}{12}.
      \]

Now consider the integral
      \begin{equation} \label{equ:integral}
         \int_{1}^{x} \left(g(\theta)-\frac{\log\theta}{12}\right)\,d\theta.
      \end{equation}
Let $1=x_0<x_1<\dots<x_N=x$ be evenly spaced from 1 to $x$. We have
      \begin{align*}
         \int_{1}^{x} \left(g(\theta)-\frac{\log\theta}{12}\right)\,d\theta
            &= \sum_{i=0}^{N-1} \int_{x_i}^{x_{i+1}} \left(g(\theta)-\frac{\log\theta}{12}\right)\,d\theta \\
            &\leq \sum_{i=0}^{N-1} \int_{x_i}^{x_{i+1}} \left(\frac{\log x_{i+1}}{12}-\frac{\log\theta}{12}\right)\,d\theta,
      \end{align*}
which goes to 0 as $N\to\infty$ since $\log\theta/12$ is Riemann integrable over $[1,x]$. Similarly, we find that~\eqref{equ:integral} is at least 0, whence $g(\theta)=\log\theta/12$ for almost all $\theta$ less than $x$, and $x$ is arbitrary.
\end{proof}

We note that LeVeque~\cite{leve76} constructed $\alpha$ with
   \[
      \limsup_{n\to\infty} \frac{1}{\log n} \sum_{k=1}^n \tf{\fp{k\alpha} < 1/k} = \infty,
   \]
showing that the ``almost all'' in Schmidt's theorem cannot be improved to ``all''. While LeVeque's construction, using continued fractions, does not immediately carry over to intervals that do not contain 0, we believe that the same phenomenon affects us. That is, we believe that for any function $g(n)\to0$, there is a $\theta$ so that $\big| A_\theta\cap[1,n] \big| > n \cdot g(n)$ for infinitely many $n$.

\section{Continued fractions and the proofs of Theorems~\ref{FPCF:theorem:explicit theta with no atypical numbers} and~\ref{FPCF:theorem:explicit theta with many atypical numbers}.}

The continued fraction algorithm produces a positive integer
from a real number $\alpha > 1$ by taking the integer part
of the reciprocal of the fractional part of $\alpha$.
This is exactly how the function $M_{\theta}(n)$ operates
on the $n$th root of a real number $\theta > 1$,
so it is, perhaps, not surprising that there is a relationship
between continued fractions and the fractional parts of roots.

We shall consider infinite continued fractions of the form
$[a_0;a_1,a_2,\ldots]$ with partial quotients $a_0 \in \Z$
and $a_k \in \N$ for all $k\in \N$.
Then $\alpha = [a_0;a_1,a_2,\ldots]$ is a real irrational number whose  \emph{$k$th convergent} is the rational number
\[
\frac{A_k}{B_k}  = [a_0;a_1,a_2,\ldots, a_k]
\]
where $A_k$ and $B_k$ are relatively prime positive integers. Also, set
  \[
   \lambda_k \coloneqq [0; a_{k-1},a_{k-2},\dots,a_1]+[a_{k};a_{k+1},a_{k+2}, \dots].
  \]
We follow the notation of Rockett and Sz\"{u}sz~\cite{rock-szus92}, and use some results that are found there but not in the other standard references.
The sequence of denominators, sometimes called \emph{continuants}, satisfies $B_k \geq F_{k+1}$, the $(k+1)$th Fibonacci number. Further,
\beq   \label{FPCF:ii}
\frac{A_{2k-2}}{B_{2k-2}}  < \frac{A_{2k}}{B_{2k}} < \alpha
<  \frac{A_{2k+1}}{B_{2k+1}}  < \frac{A_{2k-1}}{B_{2k-1}}
\eeq
and
\beq   \label{FPCF:iii}
  \alpha - \frac{A_k}{B_k}  = \frac{(-1)^k}{B_k^2 \lambda_{k+1}}.
\eeq
This is often used in conjuction with the trivial bounds
  \[
   a_{k+1} < \lambda_{k+1} < a_{k+1}+2.
  \]
If $m$ and $n$ are positive integers and
\beq     \label{FPCF:iv}
\left| \alpha - \frac{m}{n} \right| \leq \frac{1}{2n^2},
\eeq
then~\cite{rock-szus92}*{Theorem II.5.1} there are integers $k\geq0,c\geq 1$ such that $m=cA_k$ and $n=cB_k$ and $\lambda_{k+1}>2c^2$.

\begin{lemma}   \label{FPCF:lemma:atypical n are continuants}
Let $1 < \theta < e^3$ with $\log\theta$ irrational,
and $a_k, B_k, \lambda_k$ be associated to the continued fraction of $2/\log\theta$.
For each $n\in \mcA_\theta$, there exist positive integers $c,k$ such that $n=c B_{2k-1}$ and $\lambda_{2k} > \frac{6c^2}{\log \theta}$.
\end{lemma}

\begin{proof}
Let $n\in \mcA_\theta$, i.e., $M'_\theta(n) = \floor{n/\log\theta+1/2}$.
By Theorem~\ref{FPCF:lemma:SecondEquiv}
  \[
    1 - 2f\left(  \frac{\log \theta}{n} \right)
    \leq  \fp{ \frac{2n}{\log \theta} } < 1.
  \]
Let $m = 1 + \floor{2n/\log\theta}$.
Applying  the upper bound in Lemma~\ref{FPCF:lemma:f-estimates}
with $t = \log\theta/n$, we obtain
\[
0 < 1 -  \fp{ \frac{2n}{\log \theta} }
= m - \frac{2n}{\log \theta}
\leq 2f\left(  \frac{\log \theta}{n} \right)  < \frac{\log\theta}{6n}
< \frac{1}{2n}
\]
and so
\[
0 < \frac{m}{n} -  \frac{2}{\log \theta} < \frac{1}{2n^2} .
\]
Properties~\eqref{FPCF:ii} and~\eqref{FPCF:iv}
of continued fractions imply that
$m/n$ is an odd convergent to $2/\log\theta$.
Thus, there exist
positive integers $k$ and $c$ with $\lambda_{2k}>2c^2$ such that
$m = c A_{2k-1}$ and $n=c B_{2k-1}$.
It follows from property~\eqref{FPCF:iii} that
  \[
    \frac{1}{B_{2k-1}^2 \lambda_{2k}} =
    \frac{A_{2k-1}}{B_{2k-1}} -  \frac{2}{\log \theta}
    < \frac{\log\theta}{6c^2 B_{2k-1}^2}
  \]
and so $\lambda_{2k} > 6c^2/\log\theta$, which makes the earlier restriction $\lambda_{2k}>2c^2$ redundant.
This completes the proof.
\end{proof}

\begin{proof}[Proof of Theorem~\ref{FPCF:theorem:explicit theta with no atypical numbers}.] Let $a_0\geq 1$ and $a_{2k}\leq 3a_0-2$ for $k\geq1$, $\ell=[a_0;a_1,\dots]$, $\theta=e^{2/\ell}$. Then $0<\log\theta<2/a_0\leq 2$, and so $\theta$ satisfies the hypotheses of Lemma~\ref{FPCF:lemma:atypical n are continuants}. Consequently, for each $n\in\mcA_\theta$, there are positive integers $c,k$ such that $n=cB_{2k-1}$ and $\lambda_{2k}>6c^2/\log\theta$. But
   \[
      \lambda_{2k}
         = [0;a_{2k-1},a_{2k-2},\dots,a_1]+[a_{2k};a_{2k+1},\dots]
         < a_{2k}+2 \leq 3a_0
   \]
while
   \[
      \frac{6c^2}{\log\theta} = 3c^2\ell \geq 3a_0.
   \]
Therefore, there are {no} $n$ in $\mcA_\theta$.

Set $a_{2k}=1$ for $k\geq0$, and let the $a_{2k-1}$ be arbitrary positive integers, to see the first family stated in Theorem~\ref{FPCF:theorem:explicit theta with no atypical numbers}. Set $a_{2k+1}=c$, an arbitrary positive integer, for $k\geq0$ to get
   \[\ell=[1;c,1,c,1,\dots] = \frac{c+\sqrt{c(c+4)}}{2c}\]
and
   \[\theta= e^{-c+\sqrt{c(c+4)}}.\]
\end{proof}

Set
  \[
   {\mathcal Q}_\theta \coloneqq \left\{cB_{2i-1} : \text{$i$ and $c$ positive integers, $2c^2< \lambda_{2i}$} \right\}
  \]
where $B_i,\lambda_i$ correspond to the continued fraction of $1/\log\theta$ (not of $2/\log\theta$). By properties \eqref{FPCF:iii} and \eqref{FPCF:iv} of continued fractions,
   \[
      {\mathcal Q}_\theta = \left\{ n : n\geq 1, \text{ there  exists an integer $m$ with $0<m-\frac{n}{\log\theta} < \frac{1}{2n}$}\right\}.
   \]
In particular, ${\mathcal Q}_\theta$ is a set of good denominators for approximating $1/\log\theta$.

Our next lemma identifies continuants of $2/\log\theta$ that are either also good denominators for $1/\log\theta$ or are exceptional. When we apply the lemma in the proof of Theorem~\ref{FPCF:theorem:explicit theta with many atypical numbers}, we will have additional constraints that prevent the continuants from also being good denominators for $1/\log\theta$, and thereby force them to be exceptional.

\begin{lemma}   \label{FPCF:lemma:continuants can be atypical}
Let $1 < \theta < e^6$ with $\log\theta$ irrational,
and $a_k, B_k, \lambda_k$ be associated to the continued fraction of $2/\log\theta$.
For $0 < \delta < \log\theta$, choose $k_0 = k_0(\delta) \geq 3$ such that
\beq   \label{FPCF:BadIneq1}
B_{2k-1}^2 > \frac{(\log\theta)^3}{60\delta}
\eeq
for all $k \geq k_0$.
If $k\geq k_0$ and 
\beq   \label{FPCF:BadIneq2}
\lambda_{2k} \geq \frac{6}{\log\theta - \delta}
\eeq
then
   \[B_{2k-1}\in {\mathcal Q}_\theta \cup \mcA_\theta.\]
\end{lemma}

\begin{proof}
As $k \geq 3$, we have
$B_{2k-1} \geq B_{5} \geq 8$
and $0 < \log\theta/B_{2k-1} \leq 6/8 < 1$.
Continued fraction inequalities~\eqref{FPCF:ii}
and~\eqref{FPCF:iii} give
\[
0 <  \frac{A_{2k-1}}{B_{2k-1}}  - \frac{2}{\log\theta}
= \frac{1}{\lambda_{2k}B_{2k-1}^2},
\]
whence, with $\lambda_{2k} >a_{2k} \geq 1$,
    \[ 0 < A_{2k-1} - \frac{2B_{2k-1}}{\log\theta} = \frac{1}{\lambda_{2k}B_{2k-1}} < \frac 18.\]
It follows that $2B_{2k-1}/\log\theta$ is slightly less than an integer, and therefore
\beq     \label{FPCF:SimpleIneq}
\fp{  \frac{2B_{2k-1}}{\log\theta}  } = 1 - \frac{1}{\lambda_{2k}B_{2k-1}}.
\eeq
Assuming that $\lambda_{2k}$ and $B_{2k-1}$ satisfy the inequalities in~\eqref{FPCF:BadIneq1}
and~\eqref{FPCF:BadIneq2}, we have
\begin{align*}
\frac{1}{\lambda_{2k}B_{2k-1}}
& \leq \frac{1}{B_{2k-1}}\left( \frac{\log\theta - \delta}{6}  \right) \\
& = \frac{\log\theta/B_{2k-1}}{6}
- \left( \frac{60 \delta B_{2k-1}^2}{(\log\theta)^3} \right)
\frac{ (\log\theta/B_{2k-1})^3 }{360} \\
& < \frac{\log\theta/B_{2k-1}}{6} - \frac{ (\log\theta/B_{2k-1})^3 }{360}  \\
& = 2 \left( \frac{\log\theta/B_{2k-1}}{12} -  \frac{(\log\theta/B_{2k-1})^3}{720}  \right) \\
& < 2f\left( \frac{\log\theta}{B_{2k-1}}\right),
\end{align*}
where the last inequality uses the lower bound in Lemma~\ref{FPCF:lemma:f-estimates}
with $t = \log\theta/B_{2k-1}$.
Combining this with~\eqref{FPCF:SimpleIneq} gives
\[
\fp{  \frac{2B_{2k-1}}{\log\theta} } > 1 - 2 f\left( \frac{\log\theta}{B_{2k-1}}\right).
\]

If, further, $\fp{B_{2k-1}/\log\theta}<1-f(\log\theta/B_{2k-1})$, then by Lemma~\ref{FPCF:lemma:SecondEquiv}, we have $B_{2k-1}\in \mcA_\theta$. We therefore assume that
   \begin{equation} \label{FPCF:BadIneq3}
      \fp{\frac{B_{2k-1}}{\log\theta}}\geq 1-f\left(\frac{\log\theta}{B_{2k-1}}\right)
   \end{equation}
and need to show that $B_{2k-1}\in {\mathcal Q}_\theta$. Define $b_i$ through \[\frac{1}{\log\theta}=[b_0;b_1,b_2,\dots],\] and denote the convergents  of $1/\log\theta$ by $R_i/S_i$, and set
    \[\tau_k=[0;b_{k-1},\dots,b_1] + [b_k;b_{k+1},b_{k+2},\dots].\]
We need to prove that $B_{2k-1} =c S_{2i-1}$ for some $c,i \in \N$ and $2c^2<\tau_{2i}$.

Inequality~\eqref{FPCF:BadIneq3} implies that
\[
0 < 1 - \fp{  \frac{B_{2k-1}}{\log\theta} } \leq f\left(\frac{\log\theta}{B_{2k-1}}\right) < \frac{\log\theta}{12 B_{2k-1}}.
\]
Let $r = 1 + \floor{ B_{2k-1}/\log\theta }.$
Then
\[
0 < r -  \frac{B_{2k-1}}{\log\theta} < \frac{\log\theta}{12 B_{2k-1}}
< \frac{1}{2B_{2k-1}}
\]
because $\log\theta < 6$, and so
\[
0 < \frac{r}{B_{2k-1}} -  \frac{1}{\log\theta} < \frac{1}{2B_{2k-1}^2}.
\]
This implies that $r/B_{2k-1}$ is an oddth convergent to $1/\log\theta$, i.e., there are positive integers $c,i$ with $B_{2k-1} = c S_{2i-1}$ and $2c^2<\tau_{2i}$.
This completes the proof of this lemma.
\end{proof}

\begin{proof}[Proof of Theorem~\ref{FPCF:theorem:explicit theta with many atypical numbers}.]
Set $a_0=0,a_1=2$, $a_{2k}=4$ for all $k\geq1$, and let the $a_{2k+1}$ be arbitrary positive integers, giving us uncountably many options; the choices $a_{2k+1}=c$ lead to $\theta=e^{4-c+\sqrt{c(c+1)}}$. Define $\theta$ through
    \[
    \frac{2}{\log\theta} = [a_0;a_1,a_2,\dots]
    \]
which is clearly irrational, and let $B_k$ be its continuants.
Now,
    \[
    [0;2,4]<\frac{2}{\log\theta} = [ a_0;a_1,a_2,\dots ]=[0;2,4,a_3,4,a_5,\dots] < [0;2]
    \]
and so $e^4<\theta < e^{9/2}<e^6$. We take $\delta=2$, and since
   \[
      \frac{(\log\theta)^3}{60\delta} < 1
   \]
we may take $k_0=3$. As
    \[
    \lambda_{2k}>a_{2k} = 4 > \frac{6}{4-2} > \frac{6}{\log\theta-2},
    \]
Lemma~\ref{FPCF:lemma:continuants can be atypical} tells us that $B_{2k-1}$ (for $k\geq 3$) is in $\mcQ_\theta\cup\mcA_\theta$. We will show that $B_{2k-1}$ is not in $\mcQ_\theta$, and this will prove that $\mcA_\theta$ is infinite.

Let $S_k$ denote the $k$th convergent to $1/\log\theta$. Since $a_{2k}$ is always even, we have
   \[\frac{1}{\log\theta} = \frac12 \, \frac{2}{\log\theta} = \frac12\cdot [a_0;a_1,a_2,a_3,\dots] = [\frac{a_0}{2};2a_1,\frac{a_2}{2},2a_3,\dots].\]
That is, the simple continued fraction of $1/\log\theta=[b_0;b_1,\dots]$ where $b_0=0$, $b_1=4$, $b_{2k}=2$ and $b_{2k+1}=2a_{2k+1}$ for $k\geq 1$. We have
$S_0 = B_0 = 1$, $S_1=2B_1=4$, $S_2=B_2=9$, and the recursion relations for $k\geq 2$
\begin{align*}
B_{2k} & = 4B_{k-1} + B_{k-2} \\
B_{2k-1} &= a_{2k-1}B_{2k-2}+B_{2k-3} \\
S_{2k} & = 2S_{2k-1} + S_{2k-2} \\
S_{2k-1} & = 2a_{2k-1}S_{2k-2} + S_{2k-3}.
\end{align*}
These imply that  $B_{2k} = S_{2k}$ and $2B_{2k+1} = S_{2k+1}$ for all $k\geq 0$.

If $B_{2k-1}\in \mcQ_\theta$, then there are positive integers $c,i$ with $B_{2k-1}=c S_{2i-1}$ and $\tau_{2i}>2c^2$, where
   \[\tau_i \coloneqq [0,b_{i-1},\dots,b_1]+[b_i;b_{i+1},b_{i+2},\dots].\]
Clearly, $\tau_{2i}<b_{2i}+2=4$, so that necessarily $c=1$.
If $k \geq 3$ and $B_{2k-1} = S_{2 i-1}$ for some $i \geq 1$,
then $B_{2k-1} = S_{2i-1}  = 2B_{2i-1} $ and so $i < k$.
But then
   \[
   2B_{2i-1}=B_{2k-1} >B_{2k-2}+B_{2k-3} = (4B_{2k-3}+B_{2k-4})+B_{2k-3} > 5B_{2k-3}\geq 5B_{2i-1}
   \]
which is absurd. Therefore, there are no such $c,i$, and therefore $B_{2k-1}\not\in\mcQ_\theta$.
\end{proof}

\section{The identities stated in the first paragraph, $\theta=2$ and $\theta=e^{\sqrt{2}}$}
Set  $\theta = e^{\sqrt{2}}$.
Because
\[
\frac{2}{\log \theta} = \sqrt{2} = [1;2,2,2,\ldots]
\]
and for all $k \geq 1$
\[
\lambda_{2k} < 4 <  3\sqrt{2} = \frac{6}{\log \theta},
\]
Lemma~\ref{FPCF:lemma:atypical n are continuants} tells us that $\mcA_\theta$ is empty. By definition, $M'_{\theta}(n) = \floor{n/\log \theta - 1/2}$ for all $n\in\N$, and $M_\theta(n)= \floor{n/\log \theta - 1/2}$ for all $n\geq 3>\log_2\theta =\sqrt{2}/\log2$.

In the first sentence of this paper, we claimed that $M'_\theta(n)=\floor{n/\log \theta - 1/2}$ for all nonzero $n$, which we deduce now from the positive $n$ case. Assume $n>0$. Since $\floor{-x}=-\floor{x}-1$ for positive nonintegers $x$,
  \begin{align*}
    M'_\theta(-n) &= \floor{\frac{1}{\theta^{1/-n}-1}} \\
                  &= \floor{\frac{-\theta^{1/n}}{\theta^{1/n}-1}} \\
                  &= \floor{\frac{-\theta^{1/n}}{\theta^{1/n}-1}+\frac{\theta^{1/n}-1}{\theta^{1/n}-1} -1 } \\
                  &= \floor{- \frac{1}{\theta^{1/n}-1}} - 1 \\
                  &= -\floor{\frac{1}{\theta^{1/n}-1}} - 2  \\
                  &= - M'_\theta(n)-2,
  \end{align*}
making of use of the Gelfand-Schneider Theorem to be certain that
 \[\frac{1}{\theta^{1/n}-1}= \frac{1}{e^{\sqrt{2}/n}-1}\]
is not an integer. Continuing,
  \begin{align*}
      M'_\theta(-n) = -M'_\theta(n)-2
                  &= -\floor{\frac{n}{\log\theta}-\frac12}-2 \\
                  &= \floor{-\left( \frac{n}{\log\theta}-\frac12\right)} -1 \\
                  &= -\floor{ \frac{-n}{\log\theta}-\frac12},
  \end{align*}
where we have used the value and irrationality of $\log\theta=\sqrt{2}$ to guarantee that $n/\log\theta-1/2$ is positive and not an integer.
This establishes~\eqref{equ:root2} for negative $n$.

Now, set $\theta=2$, and take $n$ so that $n\in \mcA_\theta$. As $\log2$ is irrational and $\frac{2}{\log2}<3$, we can apply Theorem~\ref{FPCF:lemma:atypical n are continuants} to deduce that there are positive integers $c,k$ such that $n=cB_{2k-1}$ and $\lambda_{2k}>6c^2/\log2$ (where $B_i,\lambda_i$ correspond to the continued fraction of $2/\log2$). It is not difficult to compute $\lambda_2,\lambda_4,\dots,\lambda_{34}$ and find that only $\lambda_{2}$ is greater than $6/\log2$. Therefore, our only candidate for $\mcA_2$ less than $B_{35}=777\,451\,915\,729\,368$ is $B_1=1$ (we have to consider the multiples $cB_1$ with $8.73>\lambda_2>6c^2/\log2 > 8.65 c^2$, that is, $c=1$). Direct calculation shows that in fact $B_1$ and $B_{35}$ are both in $\mcA_\theta$. This completes our justification of the claims made in our opening paragraph. This is essentially the same as the computation of sequence A129935 in the OEIS \cite{OEIS}.

\section{More Problems}
Nathanson~\cite{Nathanson.Mtheta}*{Section 5} gives a list of problems concerning $M_\theta(n)$. Several of these problems are solved (explicitly or implicitly) in the current work, but those concerning small $n$ or letting $\theta$ vary are not addressed here. To his list, we add the following problems:

\begin{enumerate}
 \item Is $\mcA_{e^e}$ infinite?
 \item Are there $\theta,\tau$ with both $\mcA_\theta$ and $\mcA_\tau$ infinite, but the symmetric difference $\mcA_\theta \triangle \mcA_\tau$ finite? 
 \item For every $\theta_0$, are there uncountably many $\theta>\theta_0$ with $\mcA_\theta$ finite?
 \item What is the Hausdorff dimension of
   \(
      \left\{ \theta>1: \mcA_\theta\text{ is finite}\right\}?
   \)
  \item Is there any algebraic $\theta$ for which $\mcA_\theta$ can be proved finite? Infinite?
\end{enumerate}


\end{document}